\documentclass[12pt,oneside,reqno]{amsart}

\usepackage[a4paper,left=25mm,top=30mm,right=25mm,bottom=30mm]{geometry}

\usepackage{amssymb,amsfonts}
\usepackage[all,arc]{xy}
\usepackage{enumerate}
\usepackage{siunitx}
\usepackage{enumitem}
\usepackage{mathtools}
\usepackage{mathrsfs}
\usepackage{dsfont}
\usepackage{amsfonts}
\usepackage{amssymb}
\usepackage{graphicx}
\usepackage{ragged2e}
\usepackage{amsthm}
\usepackage{amsmath}
\usepackage[mathscr]{euscript}
 \let\mathscr\relax
\usepackage[scr]{rsfso}
\usepackage{graphicx}
\usepackage{color}
\usepackage{float}
\usepackage{caption}
\usepackage[pdftex,bookmarksnumbered,bookmarksopen]{hyperref}
\hypersetup{hidelinks = true}

\DeclareMathOperator{\Tr}{Tr}
\DeclareMathOperator{\vol}{Vol}
\DeclareMathOperator{\Var}{Var}

\newtheorem{thm}{Theorem}[section]

\newtheorem{prop}[thm]{Proposition}
\newtheorem{lem}[thm]{Lemma}

\theoremstyle{definition}

\theoremstyle{remark}

\numberwithin{equation}{section}

\begin{document}
	\title[Nodal lines on the square]{On the behavior of nodal lines near the boundary for Laplace eigenfunctions on the square}

\author[O. Klurman]{Oleksiy Klurman}
\address{School of Mathematics, University of Bristol, United Kingdom and Max Planck Institute for Mathematics, Bonn, Germany}
\email{lklurman@gmail.com}
\author[A. Sartori]{Andrea Sartori}
\address{School of Mathematical Sciences, Tel Aviv University, Israel}
\email{andrea.sartori.16@ucl.ac.uk}
	\maketitle
	\begin{abstract}
		We are interested in the effect of  Dirichlet boundary conditions on the nodal length of Laplace eigenfunctions.  We study random Gaussian Laplace eigenfunctions on the two dimensional square and find  a two terms asymptotic expansion for the expectation of the nodal length in any square of side larger than the Planck scale, along a denisty one sequence of energy levels. The proof relies on a new study of lattice points in small arcs, and shows that the said expectation is independent of the position of the square, giving the same asymptotic expansion both near and far from the boundaries.  
	\end{abstract}
\section{Introduction}
\subsection{Laplace eigenfunctions on plane domains}
Let $\Omega\subset \mathbb{R}^2$ be a plane domain with piece-wise real analytic boundaries, and let $\{\phi_i\}_{i\geq 1}$ be the sequence of Laplace eigenfunctions for the Dirichlet (or Neumann) boundary value problem: 
\begin{align}
	\label{problem 1}
\begin{cases}
	 \Delta \phi_i(x) +\lambda_i^2(x)=0 &x\in \Omega \\
 \phi_i(x)=0 \hspace{3mm}(\text{or}\hspace{3mm} \partial_{\nu}\phi_i(x)=0) &x\in \partial \Omega, 
\end{cases} 
\end{align} 
where $\Delta= \partial_x^2 + \partial_y^2$ for the standard two dimensional (flat) Laplace operator, $\{\lambda_i\}_{i\geq 1}$ is the discrete spectrum of $\Delta$, $\partial \Omega$ denotes the boundary of $\Omega$ and $\partial_{\nu}$ the normal derivative. The nodal set $\phi_{i}^{-1}(0)$ of an eigenfunction $\phi_i$ is a smooth curve outside a (possibly infinite) set of points \cite{C76}. We are interested in the behavior of the \textit{nodal length} $\mathcal{L}(\phi_{i}):= \mathcal{H}(\phi_{i}^{-1}(0))$, where $\mathcal{H}(\cdot)$ is the Hausdorff dimension,  near the boundary of the domain $\partial \Omega$.

Berry \cite{B1} conjectured that high energy Laplace eigenfunctions, on a chaotic surfaces, should  behave as superposition of waves with random direction and amplitude, that is as a Gaussian field $F$ with covariance function
$$\mathbb{E}[F(x)F(y)]= J_{0}(|x-y|).$$
Subsequently Berry \cite{B2} adapted to above model to predict the behavior of the nodal lines in the presence of boundaries. He considered the the random superposition of plane waves 
$$G(x_1,x_2)= \frac{1}{\sqrt{J}}\sum_{j=1}^J \sin(x_2 \sin(\theta_j))\cos(x_1\cos(\theta_j)+ \psi_j),$$
where $J\rightarrow \infty$ and the $\theta_j$ and the $\psi_j$ are random phases, so that the horizontal axis $x_2=0$ serves as a model of the boundaries. Importantly, Berry found that the density of nodal lines near the boundary is smaller than the  the density far away from the boundaries and this resulted in a (negative) logarithmic correction term on the  expected nodal length.

In the case of plane domains with piece-wise real analytic boundaries (and in the much more general case of real analytic manifolds with boundaries) Donnelly and Fefferman \cite{DF90} found $\mathcal{L}(\phi_{i})$ to be proportion to $\lambda_i$, that is
 \begin{align}
 	c_{\Omega}\lambda_i\leq \mathcal{L}(\phi_{i})\leq C_{\Omega}\lambda_i \label{yau}
 \end{align}
for  constants $c_{\Omega}, C_{\Omega}>0$, thus corroborating both Berry's predictions and a conjecture of Yau asserting that \eqref{yau} holds on any $C^{\infty}$ manifold (without boundaries). It worth mentioning that Yau'c conjecture  was established for the real analytic manifolds \cite{B78,BG72,DF}, whereas, more recently, the optimal lower bound and polynomial upper bound were proved \cite{L2,L1,LM} in the smooth case. Unfortunately, results such as \eqref{yau} do not shade any light into the possible effect of boundary conditions on the nodal length. 

In order to understand the behavior of the nodal length near the boundaries, we study random Gaussian Laplace eigenfunctions on the 2d square, also known as \textit{boundary adapted Arithmetic Random Waves}. We find  a two terms asymptotic expansion for the expectation of the nodal length in squares, of size slightly larger than $\lambda_i^{-1/2}$, the Planck scale, both near and far away from the boundaries. In both cases, the said expectation has the same  two terms asymptotic expansion showing that the effect of boundaries seem to be uniform in the whole square, even at small scales. Our findings   extend previous results obtained by Cammarota, Klurman and Wigman \cite{CKW20}, who studied the expectation of the \textit{ global} nodal length of Gaussian Laplace eigenfunctions on $[0,1]^2$, with Dirichlet boundary conditions, and complement a similar study on the two dimension round sphere by Cammarota, Marinucci and Wigman \cite{CMW20}.

\subsection{Laplace spectrum of $[0,1]^2$ }
The  Laplace eigenvalues on the square $[0,1]^2$ with Dirichlet boundary conditions are given by integer representable as the sum of two squares, that is $n\in S:=\{n\in \mathbb{Z}: n=a^2+b^2 \hspace{2mm} \text{for some} \hspace{2mm}a,b \in  \mathbb{Z}\}$ and the eigenfunctions can be written explicitly as a Fourier sum 
\begin{align}
	f_n (x)= \frac{4}{\sqrt{N}}\sum_{\substack{\xi \in \mathbb{Z}^2 \backslash \sim \\|\xi|^2=n}}a_{\xi}\sin (\pi \xi_1 x_1) \sin (\pi \xi_2 x_2) \label{function}
\end{align} 
where $N$ is the number of lattice points on the circle of radius $n$, $x=(x_1,x_2)$, $\xi=(\xi_1,\xi_2)$, the $a_{\xi}$'s are complex coefficients  and, in order to avoid repetitions, the sum is constrained by the relation $\xi\sim\eta$ if and only if $\xi_1=\pm \eta_1$ and $\xi_2= \pm \eta_2$. 

Boundary adapted Arithmetic Random Waves  (BARW) are functions $f_n$ as in \eqref{function} where the $a_{\xi}$'s are i.i.d. standard Gaussian random variables. Alternatively,  BARW are the continuous, \textit{non-stationary}, that its law of $f_n$ is not invariant under translations by elements of $\mathbb{R}^2$, Gaussian field with covariance function
\begin{align}
r_n(x,y)= \mathbb{E}[f(x)f(y)]= \frac{16}{N}\sum_{\xi\backslash\sim} \sin (\pi \xi_1 x_1) \sin (\pi \xi_2 x_2)\sin (\pi \xi_1 y_1) \sin (\pi \xi_2 y_2). \label{formula r_n}
\end{align}

For the said model,  Cammarota, Klurman and Wigman \cite{CKW20} found that the expectation of the \textit{global} nodal length of $f_n$ depends on the distribution of the lattice points on the circle of radius $n$. Explicitly, they showed that
\begin{align}
	\mathbb{E}[\mathcal{L}(f_n)]= \frac{\sqrt{n}}{2\sqrt{2}}\left( 1- \frac{1+ \hat{\nu}_n(4)}{16N} + o_{n\rightarrow \infty}\left(\frac{1}{N}\right)\right) \label{global}
\end{align}
where the limit is taken along a density one\footnote{ A sub-sequence $S'\subset S$ is of density one if $\lim\limits_{X\rightarrow \infty} |{n\in S':n\leq X}|/|{n\in S:n\leq X}|=1$. } sub-sequence of eigenvalues,
\begin{align}
	&\nu_n= \frac{1}{N}\sum_{|\xi|^2=n} \delta_{\xi/\sqrt{n}} & \hat{\nu}_n(4)= \int_{\mathbb{S}^1}\cos(4\theta) d\nu(\theta)  \label{spectral measure}
\end{align}
and $\delta_{\xi/\sqrt{n}}$ is the Dirac distribution at the point $\xi/\sqrt{n}\in \mathbb{S}^1$. Moreover, they showed that there exists subsequences of eigenvalues ${n_i}$ such that \eqref{spectral measure} holds and $\hat{\nu}_{n_i}(4)$ attains any value in $[-1,1],$ thus showing that all intermediate ``nodal deficiencies" are attainable.

Let $\mathcal{L}(f_n,s,z)=\vol\{x\in B(s,z): f_n(x)=0\}$, where  $B(s,z)$ is the box of side $s>0$ centered at the point $z\in[0,1]^2$, we prove the following: 
\begin{thm}
	\label{thm:1}
	Let $f_n$ be as in \eqref{function} and $\varepsilon>0$, then there exists a density one sub-sequence of $n\in S$ such that  
\begin{align}
	\mathbb{E}[\mathcal{L}(f_n,s,z)]= \vol (B(s,z)\cap [0,1]^2)\frac{\sqrt{n}}{2\sqrt{2}}\left( 1- \frac{1+ \hat{\nu}_n(4)}{16N} +  o_{n\rightarrow \infty}\left(\frac{1}{N}\right)\right) \label{conclusion}
\end{align}
uniformly for $s>n^{-1/2+\epsilon}$ and $z\in[0,1]^2$. Moreover, for every $a\in[-1,1]$ there exists a subsequence of eigenvalues ${n_i}$ such that $\hat{\nu}_{n_i}(4)\rightarrow a$ and \eqref{conclusion} hold.
\end{thm}

The main new ingredient in the proof of Theorem \ref{thm:1} is the study of the distribution of lattice points in small arcs, which we shall now briefly discuss.

\subsection{Semi-correlations and lattice points in shrinking sets}
\label{number theory}
The study of the nodal length in the Boundary adapted Arithmetic random wave model is intimately connected to the following general result which we shall establish in the next sections.
\begin{thm}
	\label{thm: main correlation}
	Let $\ell>0$ be an even integer and let $\overline{v}$ be any fixed directional vector in $\mathbb{R}^2.$ Let $\text{P}_{\overline{v}}:\mathbb{R}^2\to\mathbb{R}^2$ denote the operator of projection on the subspace generated by $\overline{v}.$
	Let $N=N(n)$ be the number of points $\xi_i=(\xi_i^1,\xi_i^2)\in\mathbb{Z}^2$ with $|\xi_i^1|^2+|\xi_i^2|^2=n.$ Then for any given $\varepsilon>0$  there exists a density one sub-sequence of $n\in S$ so that the inequality
	\begin{equation}\label{main correlation}
		\left| \text{P}_{\overline{v}}(\sum_{m=1}^{\ell}\xi_i)\right| \le {n^{1/2-\varepsilon}}
	\end{equation}
	has $\frac{\ell!}{2^{\ell/2}(\ell/2)!}N^{\ell/2}+o(N^{\ell/2})$ solutions.
\end{thm}
Theorem \ref{thm: main correlation} has a simple albeit important geometric interpretation: for almost all $n\in S,$ cancellations in the vector sum $\sum_{i=1}^{2\ell}\xi_i$ with $|\xi_i|^2=n$ along any given direction $\overline{v}$ can occur only for trivial reasons, namely when the last $\ell$ vectors form a cyclic permutation of the first $\ell$ vectors with opposite signs.\\
Theorem~\ref{thm: main correlation} refines and strengthens several important results in the subject. 
We highlight that in the case when $P_{\overline{v}}$ is replaced by the identity operator the same conclusion follows from the work of \cite{BMW}, which in turn generalized earlier work by Bombieri and Bourgain \cite{BB} where the right hand side of \eqref{main correlation} was assumed to be identically equal to zero. In our setting, the case when $\overline{v}=(0,1)$ and the right hand side being equal to zero has been treated in \cite{CKW20}. 

To facilitate discussion below, we let $\xi=(\xi_1,\xi_2)$ with $|\xi|^2=n$ and $\mathcal{M}^t(n,\ell)$ be the number of \textit{semi-correlations}, that is solutions to 
\begin{align}
	\xi_t^1+...+\xi_t^{\ell}=0 \label{11}
\end{align}
for $t=1,2$ and $\xi^j=(\xi_1^j,\xi_2^j)$ are representations of $n$ as the sum of two squares. We have the following result, see \cite[Theorem 1.3]{CKW20}:
\begin{lem}
	\label{lem: spectral correlations}
	Let $\ell>0$ be an even integer then for a density one of $n\in S$, we have 
	\begin{align}
		&\mathcal{M}^t(n,\ell)\leq C_{\ell}  N^{\ell/2}  &N \rightarrow \infty \nonumber
	\end{align}
	for some constant $C_{\ell}>0$ and uniformly for $t=1,2$.
\end{lem}
Of particular importance in our study of BARW will be the following special case of Theorem \ref{thm: main correlation}: let $K>0$ be a (large) parameter and let $\mathcal{V}^t(n,\ell,K)$ be the number of solutions to 
\begin{equation}\label{inequalities}
	0<|\xi_t^1+...+\xi_t^{\ell}|\leq K 
\end{equation}
for $t=1,2$, then we shall prove the following result.
\begin{thm}
	\label{thm: semi spectral correlations}
	Let $\epsilon>0$ and  $\ell>0$ be an integer. Then, for a density one of $n\in S$, we have 
	\begin{align}
		\mathcal{V}^t(n,\ell,n^{1/2-\epsilon})= \emptyset. \nonumber
	\end{align}
\end{thm}
We will also need the following simple separation result.
\begin{lem}
	\label{cor:small representations}
	Let $\epsilon>0$ then, for a density one of $n\in S$, we have 
	\begin{align}
		|\xi_t|\geq n^{1/2-\epsilon} \nonumber
	\end{align}
	for $t=1,2$ and for all $|\xi|^2=n$. 
\end{lem}

Finally, we construct sequence of circles satisfying the conclusion of Lemma \ref{cor:small representations} and Theorem \ref{thm: semi spectral correlations} for which we can control the angular distribution:  
 \begin{lem}
 	\label{lem: control lattice points}
 	Let $\epsilon>0$ be given. For any $a\in[-1,1],$ there exists a sequence $\{n_i\}_{i\ge 1},$ with $N_{n_i}\to\infty$ whenever $i\to\infty,$ such that $\widehat{\nu_{n_i}}(4)\to a$  and the conclusions of Theorem \ref{thm: semi spectral correlations}, Lemma \ref{cor:small representations} hold and moreover $\mathcal{V}^t(n_j,l,n_j^{1/2-\epsilon})=\emptyset$.  
 \end{lem}
 
\subsection{Notation}
We write $A\ll B$ or $A=O(B)$ to designate the existence of a constant $C>0$ such that $A\leq C B$, we denote the dependence the constant $C$ depends on some parameter $\ell$ say, as $A\ll_{\ell}B$. We write $B(s,z)$ for the box centered at $z\in \mathbb{T}^2$ of side length $s>0$. For two integers $m,n$, we write $m|n$ if there exists some integer $k$ such that $n=km$. 
\section{Proof of semi-correlations results}
\subsection{Number theoretic preliminaries}
We will need the following two standard results: the first is the following result due to Kubilius \cite{K} about Gaussian primes, which are primes $\mathcal{P} \subset \mathbb{Z}[i]$ such that $\mathcal{P} \cap \mathbb{Z}=p$ with $p\equiv 1 \pmod 4$. 
\begin{lem}[Kubilius]
	\label{pnt}
	Let $\theta_1,\theta_2\in [0,2\pi]$.  Then, the number of Gaussian primes in the sector $\arg(\mathcal{P})\in [\theta_1,\theta_2]$ such that $|\mathcal{P}|^2\leq X$ is 
	\begin{align}
		\frac{2}{\pi}(\theta_1-\theta_2)\int_{2}^{X}\frac{dx}{\log x} + O (X\exp(-c\sqrt{\log X})). \nonumber
	\end{align}
\end{lem}
The second  is Landau's Theorem, see for example \cite[Theorem 14.2]{FI}: there exists some explicit constant $c>0$ such that
\begin{align}
	\#\{n\in S': n\leq X\}= c \frac{X}{\sqrt{\log X}}(1+o(1)) \label{landau}.
\end{align}

\subsection{Proof of Lemma \ref{lem: control lattice points}}
\begin{proof}[Proof of Lemma \ref{lem: control lattice points}]
	Fix large $m\ge 1$ and small $\delta>0.$ We apply Lemma~\ref{pnt} to select infinite sequence of primes $p_n$ with the property that $p_{n}=\pi_{n}\bar{\pi_n},$ $p_n=1\pmod 4$ and $|\text{arg}(\pi_n)|\le \frac{\delta}{100m}.$ Furthermore, we choose large prime $p$ with
	\begin{equation}\label{fourier1}
		|\widehat{\nu_{p}}(4)-a|\le \frac{\varepsilon}{2},\end{equation}
	where $\hat{\nu}$ is as in \eqref{spectral measure} and consider the numbers of the form $n=p_n^{m}p.$ For such $p_n$ we have 
	\begin{equation}\label{fourier2}
		|\widehat{\nu_{p_n}}(4)-1|\le \frac{\delta}{2m}\end{equation}
	and $N_n> 2^m.$ 
	From the convolution identity $\widehat{\nu_{n}}(4)=(\widehat{\nu_{p_n}}(4))^m\widehat{\nu_{p}}(4)$ and the triangle inequality deduce the bound
	$$|\widehat{\nu_{n}}(4)-s|\le m|\widehat{\nu}_{p_n}(4)-1|+|\widehat{\nu}_{p}(4)-s|\le m\cdot\frac{\delta}{2m}+\frac{\delta}{2}=\delta.$$ We claim that inequality~\eqref{inequalities} has only trivial solutions for appropriately chosen values of $p_n,p,$ which satisfy~\eqref{fourier1} and~\eqref{fourier2}.\\
	To this end, we let $\pi_n=|\pi_n|e^{i\phi}$ and $p=\pi\cdot\bar{\pi}$ with $\arg{\pi}=\alpha.$
	For a given point $\xi_i$ with integer coordinates and $|\xi_i|=\sqrt{n}$ we write $\xi_j=\sqrt{n}e^{i(j\phi\pm\alpha+r\frac{\pi}{2})}$ for some $|j|\le m$ and $r=\{0,1,2,3\}.$ In these notation we rewrite~\eqref{inequalities} in the form
	\begin{align}\label{eq1}
		\left|\sum_{j=1}^{\ell}\varepsilon_j\cos \left(\l_j\phi\pm\alpha+\frac{\pi r_j}{2}\right)\right|\ll \frac{1}{n^{\epsilon}},
	\end{align}
	where $\varepsilon_j=\{+1,-1\}$ and $|\l_j|\le m$ for $1\le j\le \ell.$ The left hand side of~\eqref{eq1} can therefore be viewed as a trigonometric polynomial	\begin{equation}\label{trigonometric}
		F_{m}(\phi)=\sum_{j=1}^{r}\cos (m_j\phi)(\alpha_j\cos \alpha+\beta_j\sin \alpha)+\sin (m_j\phi)(\alpha^{(1)}_j\cos \alpha+\beta_j^{(1)}\sin \alpha),\end{equation}
	where $1\le m\le \ell$ and $0\le m_1<m_2<\dots m_r$ with $-\ell\le \alpha_j,\alpha_j^{(1)},\beta_j,\beta_j^{(1)}\le \ell$ with the constraints $|F_{m}(\phi)|\ll \frac{1}{n^{\epsilon}}.$
	For any fixed $\ell,$ there are finitely many choices for the coefficients $\alpha_j,\alpha_j^{(1)},\beta_j,\beta_j^{(1)}$
	and therefore we can select angle $\alpha$ for which the corresponding prime $p$ satisfies~\eqref{fourier1} and such that $a\sin{\alpha}+b\cos(\alpha)\ne 0$ for all $a,b\in\mathbb{Z}$ with $|a|+|b|\ne 0$ and $|a|,|b|\le \ell.$ Now since each $F_r(\phi)$ is a trigonometric polynomial of a total degree at most $2\ell,$ each non degenerate equation
	$F_{m}(\phi)=0$
	has at most $2\ell$ solutions. Therefore, the total number of solutions to all such equations is bounded in terms of $m,\ell.$ Consequently, by adjusting appropriate constants and using the uniform continuity, we may find a point $|\phi_0|<\frac{\delta}{200m},$ such that  
	\[|F_{m}(\phi)|\gg_m 1\]
	for all $|\phi-\phi_0|\ll \frac{\delta}{200m}$ uniformly for all polynomials defined above.
	By Lemma \ref{pnt}, there are infinitely many primes $p_n$ with angle $|\phi_0-\phi|\ll \frac{\delta}{200m}.$ For such prime $p_n,$ we  have 
	\[|F_{m}(\arg{\pi_n})|\gg_m 1\gg \frac{1}{n^{\epsilon}}\]
	for sufficiently large $p_n,$ which concludes the proof.
\end{proof}
\subsection{Proof of Theorem \ref{thm: semi spectral correlations}}

We begin by proving Theorem~\ref{thm: semi spectral correlations} in the case of square-free numbers. To this end, for any fixed $K\ge 1$ we introduce the pre-sieved set
\[ \Omega_{M,K}=\{n\le M,\ \text{rad}(n)=n,\ p\vert n\in S\Rightarrow\ p\ge K\},\]
where $\text{rad}(n)=\prod_{p|n} p$, that is the product over primes dividing $n$ without multiplicity, and let $\Omega_M:=\Omega_{M,1}=S\cap [1,M].$ We will need the following lemma borrowed from~\cite{BB}.
\begin{lem}
	\label{bb1}
	For $m\in \Omega_{M,K},$ let $m=p_1\cdotp_2\cdot\dots p_r$ be its factorization with $K< p_1<p_2\dots<p_r.$ Then as $M\to\infty$ we have $p_s>2^{s\Phi(s)}$ for $1\le s\le r$ holds for all $m\in\Omega_{M,K}\setminus \Omega_{M,K}^{(1)},$ where the exceptional set $\Omega_{M,K}^{(1)}$ has cardinality
	\[|\Omega_{M,K}^{(1)}|\le \eta(K, \Phi)|\Omega_{M,K}|\]
	with $\eta(K, \Phi)\to 0$ as $K\to\infty.$ If $\Phi(x)=o(\log x),$ then we can choose $\eta(K,\Phi)=K^{-1+\delta}$ for every fixed $\delta>0.$
\end{lem}
The next proposition is crucial and estimates the number of solutions~\eqref{inequalities} for almost all admissible  integers $m\in \Omega_{M,K}.$
\begin{prop}\label{squarefree}
	Let $\varepsilon,\delta>0$ be fixed. If $K\ge K(\delta)$ and $M\to\infty,$ then for all but $K^{-1+\delta}|\Omega_{M,K}|$ elements $m\in\Omega_{M,K}$ we have $\mathcal{V}^t(m,\ell,m^{1/2-\epsilon})= \emptyset.$
\end{prop}
\begin{proof} Let $\tilde{S}\subset S$ such that for every $n\in\tilde{S}$ we have $\mathcal{V}^t(m,\ell,m^{1/2-\epsilon})\ne \emptyset.$ For any prime $p$ we write $p=\pi\cdot\bar{\pi}$ where $\pi$ is the corresponding Gaussian prime with $\text{arg}(\pi)\in [0,\pi/2].$ For any integer $s\ge 1$ we introduce the set
	\[\mathcal{F}_s=\left\{n\in\Omega_{M,K},\ \omega(n)=s,\ n\in \tilde{S};\ \forall d\ne n, d\vert n\Rightarrow d\in S\setminus \tilde{S}\right\}.\]
	Fix $s\ge 1$ and consider $n\in\mathcal{F}_s$ with a given factorization $n=p_1\cdot p_2\dots p_s,$ $K<p_1<p_2<\dots<p_s.$ We have that there exist  integer points $\{\xi\}_{j=1}^{\ell}$ with $\|\xi_j\|=\sqrt{n}$ and $\varepsilon_j\in \{-1,0,1\},$ $1\le j\le \ell$ with
	\[\left|\sum_{j=1}^{\ell}\varepsilon_i\text{Re}(\xi_j)\right|\le {(p_1\cdot p_2\dots\cdot p_s)^{1/2-\varepsilon}}.\]
	Each point $\xi_r$ can be uniquely written as a product $\xi_r=i^{k_{\xi_r}}\prod_{j\le s }\pi_{j,r}^*$ where  each $\pi_{j,r}^*\in \{\pi_j,\bar{\pi}_j\}$ and $k_{\xi_r}\in\{0,1,2,3\}.$
	We now regroup the terms in the last expression by collecting $\pi_s$ and $\bar{\pi_s}$ into different  summands to end up with an equivalent form
	\begin{equation}\label{iterate1}
		\left|\text{Re}(\pi_sA_{s-1})+\text{Re}(\bar{\pi_s}B_{s-1})\right|\le {(p_1\cdot p_2\dots\cdot p_s)^{1/2-\varepsilon}} ,
	\end{equation}
	where each $A_{s-1},B_{s-1}$ consists of the sum of at most $\ell-1$ terms composed of first $(s-1)$ Gaussian primes. Let $\pi_s=|\pi_s|e^{i\phi_s},$ $A_{s-1}=|A_{s-1}|e^{ia_{s-1}}$ and $B_{s-1}=|B_{s-1}|e^{ib_{s-1}}.$ We rewrite inequality \eqref{iterate1} in the form
	\[\left||A_{s-1}|\cos (\phi_s+a_{s-1})+|B_{s-1}|\cos (b_{s-1}-\phi_s)\right|\le \frac{ {(p_1\cdot p_2\dots\cdot p_{s-1})^{1/2-\varepsilon}}}{|\pi_s|^{2\epsilon}},\] which after trigonometric manipulations simplifies to \begin{align}\label{simple_trig}
		|\cos(\phi_s)(|A_{s-1}|\cos (a_{s-1})+|B_{s-1}|\cos (b_{s-1}))&+\sin(\phi_s)(-|A_{s-1}|\sin (a_{s-1})+|B_{s-1}|\sin (b_{s-1}))| \\&\nonumber\le  \frac{ {(p_1\cdot p_2\dots\cdot p_{s-1})^{1/2-\varepsilon}}}{|\pi_s|^{2\epsilon}}.
	\end{align}
	Let $\phi_0\in [0,2\pi)$ be the angle satisfying
	\[\sin(\phi_0)=\frac{|A_{s-1}|\cos (a_{s-1})+|B_{s-1}|\cos (b_{s-1})}{((|A_{s-1}|\cos (a_{s-1})+|B_{s-1}|\cos (b_{s-1}))^2+(|A_{s-1}|\sin (a_{s-1})-|B_{s-1}|\sin (b_{s-1}))^2)^{1/2}}\]
	and 
	\[\cos(\phi_0)=\frac{-|A_{s-1}|\sin (a_{s-1})+|B_{s-1}|\sin (b_{s-1})}{((|A_{s-1}|\cos (a_{s-1})+|B_{s-1}|\cos (b_{s-1}))^2+(|A_{s-1}|\sin (a_{s-1})-|B_{s-1}|\sin (b_{s-1}))^2)^{1/2}}.\]
	With these notations~\eqref{simple_trig} implies the bound	\begin{align*}
		|&\sin (\phi_s+\phi_0)|\\&\le \frac{(p_1p_2\dots p_{s-1})^{1/2-\varepsilon}}{|\pi_s|^{2\epsilon}((|A_{s-1}|\cos (a_{s-1})+|B_{s-1}|\cos (b_{s-1}))^2+(|A_{s-1}|\sin (a_{s-1})-|B_{s-1}|\sin (b_{s-1}))^2)^{1/2}}.
	\end{align*}
	Since
	$n\in \mathcal{F}_s$ we have  $p_1p_2\dots p_{s-1}=\tilde{n}\vert n$ and so by definition $\tilde{n}\in S\setminus\tilde{S}.$ Therefore, 
	$$||A_{s-1}|\cos (a_{s-1})+|B_{s-1}|\cos (b_{s-1}))|=|\text{Re}{(A_{s-1}+B_{s-1})}| \ge  (p_1\cdot p_2\dots\cdot p_{s-1})^{1/2-\varepsilon}$$ unless $A_{s-1}=B_{s-1}=0,$ in which case we end up with a trivial solution. 
	Upon noting that denominator of the above fraction is $\ge (p_1p_2\dots p_{s-1})^{1/2-\varepsilon}$
	we deduce
	\begin{equation*}
		|\sin (\phi_s+\phi_0)|\le \frac{1}{|\pi_s|^{2\varepsilon}}\ll \frac{1}{J^{\varepsilon}},\end{equation*}
	for some $J>0$. Now for a fixed value of $\phi_0,$ by convexity we have $|\sin x|\ge \frac{2}{\pi}x$ for $x\in [0,\pi/2]$ which for $q=\{0,1\}$ yields a measure bound
	\begin{equation}\label{measure}
		|\phi_s+\phi_0+q\pi| \ll \frac{1}{|\pi_s|^{2\varepsilon}.}
	\end{equation}
	We are now ready to estimate the number of $m\in\Omega_{M,K}$ which give rise to a nontrivial solution of~\eqref{inequalities}.
	Applying Lemma~\ref{bb1} allows us to restrict to the case where $m=p_1p_2\dots p_r\in\Omega_{M,K},$ with $K<p_1<p_2<\dots<p_r$ and $p_j\ge 2^{j\Phi(j)}$ for any $1\le j\le r$ and some slowly growing function $\Phi(x)$ to be determined later.
	We observe that, for each such $m$, there exists unique $1\le s\le r$ such that the product $p_1p_2\dots p_s\in \mathcal{F}_s.$
	Given $K<p_1<p_2<\dots <p_{s-1}$ we can form at most $2^{\ell(s-1)}$ sums $A_{s-1}$ and $B_{s-1}$ and thus produce at most $2^{\ell(s-1)}$ distinct $n=p_1p_2\dots p_{s-1}p_s\in\Omega_{M,K}.$\\
	We start by partitioning the range of $p_s$ into dyadic intervals $[J,2J]$ and
	note that there are at most $J^{1-\epsilon'}$ suitable $p_s$ which satisfy \eqref{measure} for any given $\varepsilon'>\varepsilon>0.$ Indeed, we have at most $\ll \sqrt{J}$ choices for the one coordinate and $\ll J^{1/2-\epsilon'}$ choices for the other.\\
	By Lemma~\ref{bb1}, $p_s\ge \max\{2^{s\Phi(s)},K\}$ and therefore the total number of elements in $\tilde{S}\cap [1,M]$ induced by the elements in $\mathcal{F}_s$ is at most
	\begin{align}
		\ll 2^{\ell}\sum_{p_1,p_2\dots p_{s-1}}\sum_{J\le\log M}&\sum_{\max{\{K,2^{s\Phi(s})}\}\le p_s\sim J, p_1\dots p_s\in\mathcal{F}_s}\left|\left\{m\le \frac{M}{p_1p_2\dots p_{s-1}p_s},\ m\in \Omega_{M,K}\right\}\right|
		\\&\nonumber \ll 2^{ls}\sum_{J\le \log M}\sum_{\max{\{K,2^{s\Phi(s})}\}\le p_s\sim\l}\text{I}_{m/p_s\in\Omega_{K,M}}\\
		&\nonumber \ll 2^{ls}\sum_{\max{\{K,2^{s\Phi(s})}\}\le J\le \log M}\frac{M}{J\sqrt {\log {J}}}\cdot J^{1-\varepsilon}
	\end{align}
	where the last estimate comes from ``conditioning'' on at most $J^{1-\varepsilon}$ possible values of $p_s$ and the fact that $m\in\Omega_{K,M}.$ The last sum is clearly bounded above by 
	$\ll 2^{ls}\frac{M}{\max{\{K,2^{\varepsilon' s\Phi(s})}\}}.$
	Since the choice of the function $\Phi(x)$ is at our disposal as long as $\Phi(x)=o(\log x),$ we can follow the  same arguments as in \cite{BB} verbatim with $\Phi(x)$ replaced by $\varepsilon' \Phi(x)$ to arrive at the conclusion.
\end{proof}
We are now ready to handle the general case.

\begin{proof}[Proof of Theorem \ref{thm: semi spectral correlations}]
	Fix large $K>0$ and consider $\mathcal{P}_K=\prod_{p\le K}p.$ Each $n\in\mathbb{N}$ can be written in the form
	$n=n_{K}n_b$ where $(n_b,\mathcal{P}_K)=1$ and $\text{rad}(n_K)\vert \mathcal{P}_K.$ Since the number of $n\in\Omega_M$ for which $p^2|n$ for some prime $p\ge K,$ is bounded above by
	\[\sum_{p>K}|\Omega_{\frac{M}{p^2}}|\ll \sum_{p>K}\frac{M}{p^2\sqrt{\log M}}\ll \frac{M}{K\log K \sqrt{\log M}}\]
	and thus give negligible contribution.
	Consequently, we can restrict ourselves to the set of integers with $n_b$ being square-free. 
	 We now fix $n_K\in\mathbb{N}$ and count the number of $n\in \tilde{S}\cap [1,M]$ with $n_K|n$ and $(\frac{n}{n_K},\mathcal{P}_K)=1.$  More precisely, we would like to count the number of $n,$ which give nontrivial solutions to
	\[|\sum_{i=1}^{\ell}\text{Re}(\alpha_i\xi_i)|\le {n^{1/2-\varepsilon}}\]
	with $\alpha_i\vert n_K$
	and $\|\xi_i\|=\sqrt{n_b}$ for $1\le i\le \ell.$
	We now follow the proof of Proposition~\ref{squarefree} regarding $\alpha_i$ as fixed coefficients. Let $\Omega_{4,1}(n)$ denote the number of prime divisors $p=1(\bmod 4)$ of $n$ counting multiplicity. We have at most $2^{\ell\Omega_{4,1}(n_K)}$ choices for the coefficients $\alpha_i$ and so the number of $n\in \tilde{S}\cap [1,M]$ induced in this way, after appealing to Proposition~\ref{squarefree} is bounded above by
	\begin{align*}
		\sum_{\text{rad}(n_K)\vert \mathcal{P}_K}2^{\ell\Omega_{4,1}(n_K)}|\Omega_{M/(n_K),K}|&\ll \sum_{\text{rad}(n_K)\in\mathcal{P}_K}\frac{4^{k\Omega_{4,1}(n_k)}}{n_K}\left((K^{-1+\delta}+o(1))\frac{M}{\sqrt{\log M}}\right)\\&\ll \left(\frac{(\log K)^{\ell+1}}{K^{1-\delta}}+o(1)\right)\frac{M}{\sqrt{\log M}}\cdot
	\end{align*}
	The result now follows by letting $K\to\infty.$
\end{proof}

We  now briefly point out the modifications required for the proof of Theorem~\ref{thm: main correlation}.
\begin{proof}[Sketch of the proof of Theorem \ref{thm: main correlation}]
	As before, let $\ell>0$ be an even integer and let $\overline{v}$ be our directional vector in $\mathbb{R}^2$ and set $\overline{v}=|\overline{v}|e^{i\theta_v}.$
	We now follow the notations of Proposition~\ref{squarefree} and observe that, upon performing rotation by the angle $-\theta_v,$ our equation~\eqref{main correlation}
	reduces to 
	\[
	\left| \text{Re}(\sum_{m=1}^{\ell}e^{-i\theta_v}\xi_i)\right| \le {n^{1/2-\varepsilon}}.
	\]
	We can now follow the proof of the Proposition~\ref{squarefree} verbatim and note that equation~\eqref{iterate1} would now take a similar form
	\begin{equation*}
		\left|\text{Re}(e^{-i\theta_v}\pi_sA_{s-1})+\text{Re}(e^{-i\theta_v}\bar{\pi_s}B_{s-1})\right|\le {(p_1\cdot p_2\dots\cdot p_s)^{1/2-\varepsilon}}.
	\end{equation*} 
	This in turn would lead to similar expressions for $\sin (\phi_0)$ and $\cos(\phi_0)$ with the corresponding angles $a_s$ and $b_s$ replaced with $a_s-\theta_v$ and $b_s-\theta_v.$ Crucially, the bound~\eqref{measure} remains unchanged which would not affect the rest of the proof.
\end{proof}
Finally, we conclude with a short proof of Lemma \ref{cor:small representations}.
\begin{proof}[Proof of Lemma \ref{cor:small representations}]
	We partition the interval $[1,N]$ into $\sim \log N$ dyadic intervals of the form $[k,2k].$ Let $\xi=(\xi_1,\xi_2)$ with $|\xi|^2=\xi_1^2+\xi_2^2=m$ and $k\le m\le 2k.$ We observe that if $|\xi_t|\le m^{1/2-\epsilon}\le k^{1/2-\epsilon}\le 2 k^{1/2-\epsilon}$ then the number of integers $k\le m\le 2k$ is upper bounded by $\ll k^{1/2} k^{1/2-\epsilon}=k^{1-\epsilon},$ where the first factor comes from the fact that there are at most $k^{1/2}$ choices for one coordinate and at most  $k^{1/2-\epsilon}$ choices for the other. Summing over all such dyadic intervals we see that the total contribution of such $\xi$ is $\ll \log N\cdot N^{1-\epsilon}=o(N/\sqrt{\log N})$, thus Lemma \ref{cor:small representations} follows from Landau's Theorem \ref{landau}.
\end{proof}
\section{Formula for the expectation in shrinking sets}
\subsection{Deterministic grid, reduction to $n$ square-free}
Let us denote by $S':=\{n\in S: n \hspace{2mm} \text{ is square-free}\}$. In this section, we show that, in order to prove Theorem \ref{thm:1}, it is enough to restrict ourselves to $n\in S'$. More precisely, we show that  if $n$ is non-square free, then there exists a deterministic grid where $f_n(x)=0$, see also \cite{CKW20}. However, for most $n\in S$, its contribution is negligible compared to main term in Theorem \ref{thm:1}.

 To see this, let $n\in S$ and write $n= 2^{\alpha_2}\prod_j p_j^{\alpha_j} \prod_k q_k^{\beta_k}$ where $p\equiv 1 \pmod 4$, $q\equiv 3 \pmod 4$ and the $\beta_k$'s are even, and consider the \textquotedblleft fix\textquotedblright part 
\begin{align}
Q= 2^{\alpha_2}\prod_k q_k^{\beta_k} .\label{Q}
\end{align}
Then, letting $\xi=(\xi_1,\xi_2)$ be any lattice point on the circle $|\xi|^2=n$, $Q$ divides both $\xi_1$ and $\xi_2$. Therefore, $f_n$, as in \eqref{function}, vanishes  on the grid 
\begin{align}
\mathcal{G}(Q)= \cup_{k=1}^{\sqrt{Q}}\{x\in [0,1]^2: x_1= \frac{k}{\sqrt{Q}} \hspace{2mm} \text{or} \hspace{2mm} x_2= \frac{k}{\sqrt{Q}} \}.  \nonumber
\end{align}
Since, the length of the grid is 
$$\mathcal{L}(\mathcal{G}(Q))= 2 (Q-1),$$
and for almost all $n\in S$, thanks to the Erdos-Kac Theorem, see for example \cite[Part III
Chapter 3]{Tbook}, we have $Q\leq n^{1/4}\ll (\log n) ^{O(1)}$, its contribution is negligible compared to the main term in the statement of Theorem \ref{thm:1}. Hence, upon rescaling $\xi \rightarrow \xi/Q$, from now on, we assume that $n\in S'$. 
\subsection{Kac-Rice premises}
The aim of this section will be to evaluate the zero density of $f_n$ as defined in Proposition \ref{prop:Kac-Rice} (below) outside a set of \textquotedblleft singular \textquotedblright points. We begin with the following, see also \cite[Lemma 3.1]{CKW20}: 
\begin{prop}
	\label{prop:Kac-Rice}
Let $n\in S'$ and $f_n$ be as in \eqref{function}, moreover define the \textit{zero density} of $f_n$ to be 
\begin{align}
	K_1(x)= \frac{1}{(2\pi)^{1/2}\sqrt{\Var(f(x))}} \mathbb{E}[|\nabla f_n(x)||f_n(x)=0] \nonumber
\end{align}
then 
\begin{align}
\mathbb{E}[\mathcal{L}(f_n,z,s)]= \int_{B(z,s)\cap [0,1]^2} K_1(x) dx. \nonumber
\end{align}
\end{prop}
\begin{proof}
By \cite[Theorem 6.3]{AW09}, it is enough to check that the distribution $f_n(x)$ is non-degenerate for all $x\in B(z,s)$, that is 
\begin{align}
\Var(f(x))= r_n(x,x)= \frac{16}{N}\sum_{\xi\backslash\sim} (\sin (\pi \xi_1 x_1) \sin (\pi \xi_2 x_2))^2\neq 0 \label{1}
\end{align}
for all $x\in B(z,s)$. Since the left hand side of \eqref{1} is a sum of positive terms, if $\Var(f(x)) =0$  then $\xi_1 x_1= r_1\in \mathbb{Z}$ or $\xi_2 x_2= r_2 \in \mathbb{Z}$ for all $\xi$. Now, if  $\xi_1 x_1= r_1\in \mathbb{Z}$ and $\eta_2 x_2= r_2 \in \mathbb{Z}$, for $\xi\neq \eta$ then $x$ belong to a fine set of points and it does not affect the integral. If $ \xi_i x_1= r_i$ for all $\xi$ then choose  $\xi= \prod_j \mathcal{P}_j$ and $\eta= \prod_j \overline{\mathcal{P}_j}$, where $\mathcal{P}_j$ are Gaussian primes lying above the primes $p\equiv 1 \pmod 4$ dividing $n$,  to see that $x_1|Q$, with $Q$ as in \eqref{Q}. This contradicts $n$ being square-free. 
\end{proof}
In order to evaluate the zero density of $f_n$, we borrow the following lemma from from \cite[Lemma 2.2]{CKW20}:
\begin{lem}
\label{lem: expectation formula}
Let $f_n$ be as in \eqref{function} and $x\in [0,1]^2$, then
\begin{align}
\frac{2}{\pi^2 n}\mathbb{E}[\nabla f_n(x)\cdot \nabla^{t} f_n(x)|f_n(x)=0]= I_2 +\Gamma_n(x) \nonumber
\end{align}
where $\nabla^{t}$ denotes the gradient transpose, $I_2$ is the two by two identity matrix and $\Gamma_n$ is given by 
\begin{align}
	\Gamma_n = \frac{8}{n N} \begin{bmatrix}
	b_{11} &  b_{12} \\
	b_{21}   &b_{22}\\
	\end{bmatrix} - \frac{128}{n N^2 \Var (f(x))} \begin{bmatrix}
		d_1^2 &  d_1 \cdot d_2 \\
d_1 \cdot d_2   &  d_2^2\\
	\end{bmatrix} \nonumber
\end{align}
where 
\begin{align}
	&b_{11}(x)= \sum_{\xi\backslash\sim} \xi_1^2(\cos (2\pi \xi_1x_1)-\cos (2\pi \xi_2x_2)- \cos (2\pi \xi_1x_1)\cos (2\pi \xi_2x_2) ) \nonumber \\
	&b_{12}(x)= b_{21}(x)= \sum_{\xi\backslash\sim} \xi_1 \xi_2 \sin(2\pi \xi_1x_1) \cdot \sin(2\pi \xi_2x_2) \nonumber \\
	&b_{22}(x)= \sum_{\xi\backslash\sim} \xi_2^2(\cos (2\pi \xi_1x_1)-\cos (2\pi \xi_2x_2)- \cos (2\pi \xi_1x_1)\cos (2\pi \xi_2x_2) ) \nonumber \\
	&d_1(x)= \sum_{\xi\backslash\sim}\xi_1 \sin(2\pi \xi_1x_1) \cdot \sin(\pi \xi_2x_2)^2 \nonumber \\ 
	&d_2(x)= \sum_{\xi\backslash\sim}\xi_2 \sin(2\pi \xi_2x_2) \cdot \sin(\pi \xi_1x_1)^2 \nonumber. \\
\end{align}
\end{lem}
\subsection{The singular set}
\label{sec:singular set}
Let $f_n$ be as in \eqref{function}, $z\in [0,1]^2$, $s>0$ and $\Gamma_n$ be as in Lemma \ref{lem: expectation formula}. In this section, we want to bound the contributions to $\mathbb{E} [\mathcal{L}(f_n,z,s)]$ coming from points $x\in B(z,s)$ where $\Gamma_n$ is somewhat \textquotedblleft large \textquotedblright. More precisely, we  divide $B(z,s)$ into $O((s\cdot n^{1/2})^2/\delta^2)$ squares $Q_i$ of size $\delta/\sqrt{n}$ for some parameter $\delta>0$ to be chosen later, and say that square $Q_i$ is singular if it contains a point such that 
\begin{align}
\Var (f(x))=: 1- s_n(x)> 1-\gamma \hspace{5mm} \textit{or} \hspace{5mm} |\Tr(\Gamma_n)|\geq \gamma \hspace{5mm} \textit{or} \hspace{5mm} |\det(\Gamma_n)|\geq \gamma \nonumber
\end{align}
for some $\gamma>0$ to be chosen later. We denote by $\mathcal{Q}_{\text{sing}}$ the union of the singular $\mathcal{Q}_i$. 
We then prove the following proposition:
\begin{prop}
\label{prop: singular set}
Let $f_n$ be as in \eqref{function}, $\ell>0$ be an even integer and $K_1$ be as in Proposition \ref{prop:Kac-Rice}. Then, for a density one of $n\in S'$ we have 
$$ \int_{\mathcal{Q}_{\text{sing}}} K_1(x)  dx \ll\frac{s^2\sqrt{n}}{N^{\ell/2-1}}.$$
\end{prop}
In order to prove Proposition \ref{prop: singular set} we will need two lemmas. The first is the following deterministic bound on the nodal set of Laplace eigenfunctions on the square, see  \cite[Proposition 1.5]{S20}:
\begin{lem}
	\label{lem: bound on nodal length}
Let $f_n$ be as in \eqref{function} and $\varepsilon>0$, then
\begin{align}
	\mathcal{L}(f_n, z,s )s^{-1}\ll s\sqrt{n} + N  \nonumber
\end{align}
uniformly for all $z\in [0,1]^2$ and $s>0$. 
\end{lem}
The second is an estimate on the size of the singular set as follows:
\begin{lem}
	\label{lem: singular set}
Let	$\mathcal{Q}_{\text{sing}}$ be as at the beginning of section \ref{sec:singular set} and and $\ell>0$ be an even integer, then we have  
\begin{align}
	\vol( \mathcal{Q}_{\text{sing}})\ll_{\ell,\gamma} s^2 \max_{t=1,2}\left( \frac{ \mathcal{M}^t(n,\ell) }{N^{\ell}} +  \frac{1}{N^{\ell}}\sum_{\substack{\xi^1,...,\xi^{\ell}\\ \sum_i \xi_t^i\neq 0}}  \frac{1}{s|\sum_i \xi_t^i|}  \right), \nonumber
\end{align}
where the sum is subject to the relation $\xi\sim \eta$ and $M^t(n,\ell)$ is as in Lemma \ref{lem: spectral correlations}.
\end{lem}
\begin{proof}
Let us first consider squares $\mathcal{Q}_i$ where $\Var(f(x))> 1-\gamma$, that is $|s_n(x)|\leq \gamma$. Since $|\nabla s|\leq 100 \sqrt{n}$ and $\mathcal{Q}_i$ has size $\delta/\sqrt{n}$, choosing $\delta$ sufficiently small depending on $\gamma$, we may assume that $|s_n(x)|\leq \gamma/2$ for all $x\in \mathcal{Q}_i$. Thus,  by Chebyshev's bound, for any even integer $\ell>0$, we have 
\begin{align}
	\vol \{x\in B(z,s)\cap [0,1]^2: |s_n(x)|\geq \gamma/2\}\leq \frac{2^{\ell}}{\gamma^{\ell}}\int_{B(z,s)\cap [0,1]^2} |s_n(x)|^{\ell} dx. \nonumber
\end{align} 
Re-writing the definition of $r_n(x,x)$ in \eqref{formula r_n} using $\sin (x)^2 = (1/2)(1-\cos2x)$, we see that  
\begin{align}
s_n(x)= \frac{4}{N}\sum_{\xi\backslash\sim} (\cos (2\pi \xi_1x_1)+\cos (2\pi \xi_2x_2)- \cos (2\pi \xi_1x_1)\cos (2\pi \xi_2x_2) ) =: A_n(x_1)+ B_n(x_2)+ C_n(x). \nonumber 
\end{align}
Thus, since $(a+b+c)^{\ell}\ll_{\ell} a^{\ell}+ b^{\ell}+ c^{\ell}$, we have 
\begin{align}
\int_{B(z,s) \cap [0,1]^2} |s_n(x)|^{\ell}dx \ll_{\ell} \int_{B(z,s)\cap [0,1]^2} A_n(x_1)^{\ell} + B_n(x_2)^{\ell} + C_n(x)^{\ell} dx. \nonumber
\end{align}

Let us consider $A_n(x_1)$, using the transformation $x\rightarrow z+sy$, we have 
\begin{align}
\int_{B(z,s)\cap [0,1]^2} A_n(x_1)^{\ell}dx\ll s^2\int_{a(z)}^{b(z)} A_n(z_1+ sy_1)^{\ell}dy_1\leq s^2\int_{-1/2}^{1/2} A_n(z_1+ sy_1)^{\ell}dy_1 , \label{2} 
\end{align}
where $s^{-1}|a(z)|,s^{-1}|b(z)|\leq 1/2$ correspond to the projection of $B(z,s)\cap[0,1]^2$ along the $X$-axis.  Moreover, using the formula $\cos(x+y)= \cos(x)\cos(y)- \sin(x)\sin(y)$, we may write
\begin{align}
	A_n(z_1+sy_1)= \frac{1}{N}\sum_{\xi\backslash\sim} \cos(2\pi \xi_1 z_1)\cos(2\pi \xi_1 s y_1)  - \sin(2\pi \xi_1 z_1)  \sin(2\pi \xi_1 sy_1)\nonumber  \\ 
	=: A_n^1(z_1+sy_1)+ A_n^2(z_1+sy_1). \nonumber     
\end{align}
Thus, using the fact that $(a+b)^{\ell}\ll_{\ell} a^{\ell}+ b^{\ell}$, the RHS of \eqref{2} can be bounded by 
\begin{align}
\int_{-1/2}^{1/2}	A_n(z+ sy_1)^{\ell}dy_1 \ll_{\ell} \int_{-1/2}^{1/2} A_n^1(z+sy_1)^{\ell} + A_n^2(z+sy_1)^{\ell}dy_1. \label{3}
\end{align}
Let us consider $A_n^1$, expanding the $\ell$-th power, we have 
\begin{align}
	\int_{-1/2}^{1/2} A_n^1(z+sy_1)^{\ell} dy_1= \frac{1}{N^{\ell}} \sum_{\xi^1,...,\xi^{\ell}} \prod_{i=1}^{\ell} \cos (2\pi \xi_1^iz_1) \int_{-1/2}^{1/2}  \prod_{i=1}^{\ell}\cos (2\pi \xi_1^i sy_1) dy_1. \nonumber \\
	\leq \frac{1}{N^{\ell}} \sum_{\xi^1,...,\xi^{\ell}} \left|\int_{-1/2}^{1/2}  \prod_{i=1}^{\ell}\cos (2\pi \xi_1^i sy_1) dy_1 \right|\label{4}
\end{align}
Thanks to the formula $2^{\ell-1}\prod_{i=1}^{\ell} \cos (a_i)= \sum_{v\in \{-1,1\}^\ell} \cos \left( \sum_i v_ia_i\right)$ which follows by induction using the formula $\cos(a\cdot b)=(1/2)(\cos (a+b) + \cos (a-b))$, the inner integral on  the right hand side of \eqref{4} can be rewritten as 
\begin{align}
\int_{-1/2}^{1/2}  \prod_{i=1}^{\ell}\cos (2\pi \xi_1^i sy_1) dy_1 = 2^{-\ell+1}\sum_{v\in \{-1,1\}^{\ell}} \int_{-1/2}^{1/2} \cos \left(2\pi s y_1\left( \sum_{i=1}^{\ell} v_i \xi_1^i\right) \right) dy_1 \nonumber
\end{align}
Separating the terms with $\sum_i v_i \xi_1^i=0$ from the others, bearing in mind that the sum is over $\ell$-tuples satisfying the congruence relation $\xi\sim \eta$ if and only if $\xi_1=\pm \eta_1$ and $\xi_2=\pm \eta_2$, we have 
\begin{align}
\sum_{\xi^1,...,\xi^{\ell}}	 \left|\int_{-1/2}^{1/2} \cos \left(2\pi s y_1\left( \sum_{i=1}^{\ell} v_i \xi_1^i\right) \right) dy_1\right|\leq \mathcal{M}^1(n,{\ell}) + \sum_{ \sum \xi_i\neq 0}\left|  \int_{-1/2}^{1/2} \cos \left(2\pi s y_1\left( \sum_{i=1}^{\ell} v_i \xi_1^i\right) \right) dy_1\right|  \nonumber \\
= \mathcal{M}^1(n,{\ell}) + \sum_{\substack{\xi^1,...,\xi^{\ell}\\ \sum_i \xi_1^i\neq 0}} O \left( \frac{1}{s|\sum_i \xi_1^i|}\right) \label{5}
\end{align}
uniformly for all choices of $v\in \{-1,1\}^{\ell}$.  Thus, inserting \eqref{5} into \eqref{4}, we obtain 
\begin{align}
	\int_{-1/2}^{1/2} |A_n^1(z+sy_1)|^{\ell} dy_1\leq  \frac{ \mathcal{M}^1(n,{\ell}) }{N^{\ell}} + \frac{1}{N^{\ell}}\sum_{\substack{\xi^1,...,\xi^{\ell}\\ \sum_i \xi_1^i\neq 0}} O \left( \frac{1}{s|\sum_i \xi_1^i|}\right) .\label{6}
\end{align}
 Inserting \eqref{6} into \eqref{3} and using a similar argument to bound the contribution from $A_n^2(z+sy)$ we have 
\begin{align}
\int_{-1/2}^{1/2}	A_n(z+ sy_1)^{\ell}dy_1 \ll_{\ell} s^2 \max_{t=1,2}\left( \frac{ \mathcal{M}^t(n,\ell) }{N^{\ell}} +  \frac{1}{N^{\ell}}\sum_{\substack{\xi^1,...,\xi^{\ell}\\ \sum_i \xi_1^i\neq 0}} O \left( \frac{1}{s|\sum_i \xi_1^i|}\right)\right). 
\end{align}

A similar argument bounds the contribution form $B_n(x)$ and $C_n(x)$. Therefore, all in all, we have shown that 
 \begin{align}
 \int_{B(z,s)\cap [0,1]^2} |s_n(x)|^{\ell}dx \ll_{l,\epsilon} s^2 \max_{t=1,2}\left( \frac{ \mathcal{M}^t(n,{\ell}) }{N^{\ell}} + \frac{1}{N^{\ell}}\sum_{\substack{\xi^1,...,\xi^{\ell}\\ \sum_i \xi_t^i\neq 0}} \frac{1}{s|\sum_i \xi_t^i|}  \right) \label{final bound}
 \end{align}

We are left with considering squares with $|s_n(x)|\leq \gamma$, but $|\Tr(\Gamma_n)|\geq \gamma$ or $|\det(\Gamma_n)|\geq \gamma$. Again by Chebyshev's bound, for any $\ell>0 \in$ even, we have 
\begin{align}
	\vol \{x\in B(z,s)\cap[0,1]^2: |\Tr(\Gamma_n)|\geq \gamma/2\}\leq \frac{2^{\ell}}{\gamma^{\ell}}\int_{B(z,s)\cap [0,1]^2} |\Tr(\Gamma_n)|^{\ell} dx. \nonumber
\end{align}  
However, as we may assume that $|s_n(x)|\leq \gamma$,  we perform the asymptotic expansion 
$$ \frac{1}{1-s_n(x)}= 1+O\left( s_n(x)^2\right),$$
in the formula for $\Tr(\Gamma_n)$ and observe that bounding moments of  $|\Tr(\Gamma_n)|$ again reduces to computations similar to moments of $s_n(x)$, which we therefore obit. Similarly, we can bound moments of $|\det(\Gamma_n)|$ and Lemma \ref{lem: singular set} follows from \eqref{final bound}.
\end{proof}
We are finally ready to prove Proposition \ref{prop: singular set}:
\begin{proof}[Proof of Proposition \ref{prop: singular set}]
Let $Q$ be a singular square, then Proposition \ref{prop:Kac-Rice} and Lemma \ref{lem: bound on nodal length}, applied with $s=n^{-1/2}\delta$, imply that 
$$\int_Q K_1(x)dx= \mathbb{E}[\mathcal{L}(f_n,Q)]\lesssim \frac{\delta N}{\sqrt{n}}.$$ 
Thus, Lemma \ref{lem: singular set}, bearing in mind that each singular square is counted $n$-times, and taking $\delta,\gamma>0$ to be two small, fixed constants, gives 
\begin{align}
	\label{2.3.1}\int_{\mathcal{Q}_{\text{sing}}} K_1(x)dx\ll s^2 \frac{\sqrt{n}}{N^{\ell-1}}\max_{t=0,1} \left( \mathcal{M}^t(n,\ell) + \sum_{\substack{\xi^1,...,\xi^{\ell}\\ \sum_i \xi_t^i\neq 0}}  \frac{1}{s|\sum_i \xi_t^i|}  \right)
	\end{align}
where $\ell>0$ is an even integer. Hence, in light of the fact that $N\ll n^{\varepsilon}$ for all $\varepsilon>0$, Proposition \ref{prop: singular set} follows from \eqref{2.3.1} together with  Lemma \ref{lem: spectral correlations} and Theorem \ref{thm: semi spectral correlations}. 
\end{proof}
\section{Proof of Theorem \ref{thm:1}}
In order to complete the proof of Theorem \ref{thm:1}, we need to evaluate the integral of $K_1$, as in Proposition \ref{prop:Kac-Rice}, outside the singular set. This will be the content of the next section: 
\subsection{Asymptotic expansion outside the singular set}
The following proposition follows from Lemma \ref{lem: expectation formula} and a standard calculations about the expectation of a two dimensional Gaussian random variable, see also \cite[Proposition 2.7]{CKW20}:
\begin{prop}
	\label{prop: asymptotic zero density} 
	Let $K_1(x)$ be as in Proposition \ref{prop:Kac-Rice}, then for $x\in [0,1]^2\backslash \mathcal{Q}_{\text{sing}}$ we have 
	\begin{align}
K_1(x)= \frac{ \sqrt{n}}{2\sqrt{2}} + L_n(x) + \Upsilon_n(x) \nonumber
	\end{align}
	where 
	\begin{align}
		L_n(x)= \frac{\pi \sqrt{n}}{4\sqrt{2}}\left( s_n(x)+ \frac{\Tr \Gamma_n}{2}+ \frac{3}{4}s_n^2 + \frac{1}{4}s_n(x) \cdot \Tr \Gamma_n - \frac{\Tr (\Gamma_n^2)}{16}- \frac{(\Tr \Gamma_n)^2}{32}\right) \nonumber
	\end{align}
	and 
	\begin{align}
|\Upsilon_n(x)|\ll \sqrt{n} \left(|s_n(x)|^3 + |\Gamma_n(x)|^3\right) \nonumber
	\end{align}
\end{prop}
Therefore, in order to evaluate the integral of $K_1$, we will need the following lemma:
\begin{lem}
	\label{lem evaluation averages}
Let $\varepsilon>0$ and write $\mathcal{S}= \vol \left( B(z,s)\cap [0,1]^2\right)$. There exists a density one of $n\in S'$ such that 
	\begin{align}
		&\mathcal{S}^{-1}\int_{B(z,s)\cap [0,1]^2} L(x)dx= -\frac{\pi(1+\hat{\nu_n}(4))}{32\sqrt{2}}\cdot \frac{\sqrt{n}}{N} + O\left( \frac{\sqrt{n}}{N^2}\right) \nonumber  \\
		& \mathcal{S}^{-1}\int_{B(z,s)\cap [0,1]^2} |\Upsilon_n(x)|dx\ll  \frac{\sqrt{n}}{N^2} \nonumber	
	\end{align}
uniformly for all $s>n^{-1/2+\varepsilon}$ and $z\in [0,1]^2$. 
\end{lem}

We observe that Lemma \ref{lem evaluation averages} follows from the following lemma via an immediate computation:
\begin{lem}
	\label{lem: calculations}
Let $\epsilon>0$ and write $\mathcal{S}= \vol \left( B(z,s)\cap [0,1]^2\right)$.  There exists a density one of $n\in S'$ such that 
	\begin{enumerate}
	\item $$ \mathcal{S}^{-1}\int_{B(z,s)\cap [0,1]^2} s_n(x)dx \ll n^{-\varepsilon/2} $$
	\item $$ \mathcal{S}^{-1}\int_{B(z,s)\cap [0,1]^2} \Tr \Gamma_n(x)dx =  -\frac{6}{N} + O\left(N^{-2}\right) $$
	\item $$\mathcal{S}^{-1}\int_{B(z,s)\cap [0,1]^2} s_n^2(x)dx=  \frac{5}{N} + O(n^{-\epsilon/2})$$
	\item $$ \mathcal{S}^{-1}\int_{B(z,s)\cap [0,1]^2}s_n(x) \cdot \Tr \Gamma_n(x) dx= \frac{2}{N}+ O\left( N^{-2}\right) $$
	\item $$\mathcal{S}^{-1}\int_{B(z,s)\cap [0,1]^2} \Tr (\Gamma_n^2(x)) dx = \frac{4}{N}[1+2^5 \sum_{\xi\backslash\sim} \xi_1^4]  + O\left( N^{-2}\right)$$
	\item $$ \mathcal{S}^{-1}\int_{B(z,s)\cap [0,1]^2} (\Tr \Gamma_n(x))^2dx =\frac{4}{N}[2^6 \sum_{\xi\backslash\sim} \xi_1^4-3] +O\left(N^{-2}\right) $$
	\item $$ \mathcal{S}^{-1}\int_{B(z,s)\cap [0,1]^2} s_n^3(x)dx \ll N^{-2} $$
	\item $$ \mathcal{S}^{-1}\int_{B(z,s)\cap [0,1]^2} (\Tr \Gamma_n(x))^3 dx \ll N^{-2} $$
	\end{enumerate}
uniformly for all $s>n^{-1/2+\varepsilon}$ and $z\in [0,1]^2$. 
\end{lem}
Indeed, we have the following: 
\begin{proof}[Proof of Lemma \ref{lem evaluation averages} given Lemma \ref{lem: calculations}]
	By Lemma \ref{lem: calculations}, with the same notation, we
	\begin{align}
	&\mathcal{S}^{-1}\int_{B(z,s)\cap [0,1]^2} s_n^3(x)dx \ll N^{-2} 	& \mathcal{S}^{-1}\int_{B(z,s)\cap [0,1]^2} (\Tr \Gamma_n(x))^3 dx \ll N^{-2}. \nonumber
	\end{align} 
and the second part of Lemma \ref{lem evaluation averages} follows upon noticing that, using the inequality $|ab|\leq a^2+b^2$, the off diagonal entries of $\Gamma_n$ can be bounded by the diagonal ones. The first part of Lemma \ref{lem evaluation averages} follows from the remaining asymptotic formulas in  Lemma \ref{lem: calculations} together with the following identity: 
$$11- 2^7 \sum_{\xi\backslash \sim} \xi_1^4= \hat{\nu_n}(4).$$
\end{proof}
The proof of Lemma \ref{lem: calculations}, being similar to the proof of Lemma \ref{lem: singular set} will be given in Appendix \ref{appendix}.
\subsection{Proof of Theorem \ref{thm:1}}
We are now in the position to prove Theorem \ref{thm:1}
\begin{proof}[Proof of Theorem \ref{thm:1}]
Let $\varepsilon>0$ and $\ell>0$ be an even integer, thanks to Proposition \ref{prop:Kac-Rice} and Proposition \ref{prop: singular set}, with the same notation, uniformly for all $s>n^{-1/2+\varepsilon}$ and $z\in [0,1]^2$, we have  
	\begin{align}
\label{3.2.1}	\mathbb{E}[\mathcal{L}(f_n,z,s)]&= \int_{B(z,s)\cap [0,1]^2} K_1(x)dx= \int_{B(z,s)\cap [0,1]^2\backslash \mathcal{Q}_{\text{sing}}} K_1(x)dx + \int_{\mathcal{Q}_{\text{sing}}}K_1(x)dx \nonumber \\
		&=  \int_{B(z,s)\cap [0,1]^2\backslash \mathcal{Q}_{\text{sing}}} K_1(x)dx + O\left(\frac{s^2 n^{1/2}}{N^{\ell/2-1}}\right).
	\end{align}
 Using Proposition \ref{prop: asymptotic zero density}, with the same notation,  we have 
\begin{align}
	\label{3.2.2}
	\int_{B(z,s)\cap [0,1]^2\backslash \mathcal{Q}_{\text{sing}}} K_1(x)dx =  	\int_{B(z,s)\cap [0,1]^2\backslash \mathcal{Q}_{\text{sing}}} \left(\frac{\sqrt{n}}{2\sqrt{2}} + L_n(x) +\Gamma_n(x)\right)dx
	\end{align}
Assuming the conclusion of Lemma \ref{lem: spectral correlations} and Theorem \ref{thm: semi spectral correlations}, bearing in mind that $s_n(x)=O(1)$ and $\Gamma_n(x)=O(1)$, thanks to Lemma \ref{lem: singular set}, we may extend the integral on the RHS of \eqref{3.2.2} to the whole of $B(s,z)\cap [0,1]^2$ at a cost of an error term of size at most $O_{\delta}\left(\frac{s^2 n^{1/2}}{N^{\ell/2-1}}\right)$  to find 
	\begin{align}
	\int_{B(z,s)\cap [0,1]^2\backslash \mathcal{Q}_{\text{sing}}} K_1(x)dx = 	\vol\left( B(z,s)\cap [0,1]^2\right)\frac{\sqrt{n}}{2\sqrt{2}}\left( 1- \frac{1+ \hat{\nu}_n(4)}{16N} +  o\left(\frac{1}{N}\right)\right) + O\left(\frac{s^2\sqrt{n}}{N^{\ell/2-1}}\right). \label{3.2.3}
\end{align} 
Hence, Theorem \ref{thm:1} follows upon inserting \eqref{3.2.3} into \eqref{3.2.1} and taking $\ell=6$, say.   
\end{proof}

\section*{Acknowledgment.}
We thank Igor Wigman for suggesting the problem under consideration and the many useful discussions. A. Sartori was supported by the Engineering and Physical Sciences Research Council [EP/L015234/1], the ISF Grant
1903/18 and the IBSF Start up Grant no. 201834. O. Klurman greatly acknowledges the support and excellent working conditions at the Max Planck Institute
for Mathematics (Bonn) and Oberwolfach Research Institute for Mathematics (MFO).

\appendix 
\section{Proof of Lemma \ref{lem: calculations}}
\label{appendix}
To prove Lemma \ref{lem: calculations} we will use the following: 
\begin{lem}
	\label{lem: moments of cos} Let $s>0$, $z\in [0,1]^2$, $i=1,2$ and write $\mathcal{S}= \vol \left( B(z,s)\cap [0,1]^2\right)$. We have the following bounds:
	\begin{enumerate}
	\item $$\sum_{\xi\backslash\sim}	\mathcal{S}^{-1}\int_{B(z,s)\cap [0,1]^2} \cos(2\pi \xi_i x_i)dx \ll \sum_{\xi\backslash\sim} \frac{1}{|\xi_i s|} $$
		\item $$\sum_{\xi\backslash\sim}	\mathcal{S}^{-1}\int_{B(z,s)\cap [0,1]^2} \cos(2\pi \xi_i x_i)^2dx = \frac{1}{2}\cdot\frac{N}{4} +O\left( \sum_{\xi\backslash\sim} \frac{1}{|\xi_i s|}\right)$$
			\item $$\sum_{\xi,\eta}	\mathcal{S}^{-1}\int_{B(z,s)\cap [0,1]^2} \cos(2\pi \xi_i x_i)\cos(2\pi \eta_i x_i)dx =\frac{1}{2}\cdot\frac{N}{4} +O\left(\frac{1}{\sum_{\xi\backslash\sim}|\xi_i s|} +\sum_{\xi\neq\eta} \frac{1}{s|\xi_i-\eta_i|}\right) $$
			\item $$\sum_{\xi\backslash\sim}	\mathcal{S}^{-1}\int_{B(z,s)\cap [0,1]^2} \cos(2\pi \xi_i x_i)^3dx \ll \sum_{\xi\backslash\sim}|\xi_i s| $$
				\item $$\sum_{\xi\backslash\sim}	\mathcal{S}^{-1}\int_{B(z,s)\cap [0,1]^2} \cos(2\pi \xi_i x_i)^4dx = \frac{3}{8}\cdot \frac{N}{4} +O\left(\sum_{\xi\backslash\sim}\frac{1}{|\xi_i s|}\right)$$
	\end{enumerate}
\end{lem}
\begin{proof}

Through the proof we will write $\mathcal{B}=\tilde{B}(z)= s^{-1}\cdot( B(z,s)\cap[0,1]^2-z)$, that the image of $ B(z,s)\cap[0,1]^2$ under the homothety defined by scaling by translation by $-z$ and scaling by $s^{-1}$. Moreover, we denote by $a_i=a_i(z)$ and $b_i=b_i(z)$ for $i=1,2$ the coordinate of the projection of (the corners of) $\mathcal{B}(z)$ along the $X$ and $Y$ axis respectively. Using the transformation $x\rightarrow z+sy$, we have 
\begin{align}
	\sum_{\xi\backslash\sim}	\mathcal{S}^{-1}\int_{B(z,s)\cap[0,1]^2} \cos(2\pi \xi_i x_i)dx = \frac{s^2}{\mathcal{S}}\sum_{\xi\backslash\sim} \int_{\tilde{B}} \cos(2\pi \xi_i(z_i+sy_i))dy_i \nonumber 
	\end{align}
	Since $s^2/\mathcal{S}=O(1)$,	using the formula $\cos(a+b)=\cos(a)\cos(b)- \sin(a)\sin(b)$, we obtain  
	\begin{align}
	 \sum_{\xi\backslash\sim}	\mathcal{S}^{-1}\int_{B(z,s)\cap[0,1]^2} \cos(2\pi \xi_i x_i)dx \ll  \sum_{\xi\backslash\sim} \left|\int_{a_i}^{b_i}\cos(2\pi \xi_i sy_i)dy_i\right| +  \sum_{\xi\backslash\sim} \left|\int_{a_i}^{b_i}\sin(2\pi \xi_i sy_i)dy_i\right| \nonumber \\
	 \ll \sum_{\xi\backslash\sim} \frac{1}{|\xi_i s|}, \nonumber
\end{align}	 
this concludes the proof of $(1)$. 

Using the formula $2\cos(a)^2=1+\cos(2a)$, we may rewrite $(2)$ as 
\begin{align}
\sum_{\xi\backslash\sim}	\mathcal{S}^{-1}\int_{B(z,s)\cap [0,1]^2} \cos(2\pi \xi_i x_i)^2dx= \frac{1}{2}\cdot\frac{N}{4} + \sum_{\xi\backslash\sim}\int_{B(z,s)\cap [0,1]^2} \cos(4\pi \xi_i x)dx \nonumber \\
 = \frac{1}{2}\cdot\frac{N}{4} +O\left( \sum_{\xi\backslash\sim} \frac{1}{|\xi_i s|}\right),\label{7.2}
\end{align}
 where we have bounded the error term using a similar bound to the one used to obtain $(1)$. This proves $(2)$.
 
Separating diagonal terms from the others, and using \eqref{7.2}, $(3)$  becomes 
\begin{align}
	\sum_{\xi,\eta}	\mathcal{S}^{-1}\int_{B(z,s)\cap [0,1]^2} \cos(2\pi \xi_i x_i)\cos(2\pi \eta_i x_i)dx= \nonumber \\ =\frac{1}{2}\cdot\frac{N}{4} +	\sum_{\xi\neq\eta}	\mathcal{S}^{-1}\int_{B(z,s)\cap [0,1]^2} \cos(2\pi \xi_i x_i)\cos(2\pi \eta_i x_i)dx+O\left( \sum_{\xi\backslash\sim} \frac{1}{|\xi_i s|}\right) .\label{7.3}
\end{align}
Using the formula $2\cos(a)\cos(b)= \cos(a+b)+ \cos(a-b)$, the second term in \eqref{7.3} is at most
\begin{align}
	\sum_{\xi\neq\eta}\left|\int_{\tilde{B}} \cos(2\pi (\xi_i +\eta_i)(z_i+sy_i))dy\right|+\left|\int_{\tilde{B}} \cos(2\pi (\xi_i -\eta_i)(z_i+sy_i))dy\right|
\nonumber \\	\ll \sum_{\xi\neq\eta} \frac{1}{s|\xi_i-\eta_i|}, \nonumber
\end{align}
this concludes the proof of $(3)$.

Writing $4\cos(a)^3= 3\cos(a)+ \cos(3a)$, $(4)$ becomes  
\begin{align}
	\sum_{\xi\backslash\sim}	\mathcal{S}^{-1}\int_{B(z,s)\cap [0,1]^2} \cos(2\pi \xi_i x_i)^3dx \ll \sum_{\xi\backslash\sim}\frac{1}{|\xi_i s|}, \nonumber
\end{align}
this concludes the proof of $(4)$. 

Finally, using the fact that $8\cos(a)^4= 2(1+\cos(2a))^2= 3+ 4\cos(2a)+ \cos(4a)$ we obtain 
\begin{align}
\sum_{\xi\backslash\sim}	\mathcal{S}^{-1}\int_{B(z,s)\cap [0,1]^2} \cos(2\pi \xi_i x_i)^4dx= \frac{3}{8}\cdot \frac{N}{4} +O\left(\sum_{\xi\backslash\sim}\frac{1}{|\xi_i s|}\right), \nonumber
\end{align}
as required. 
\end{proof}
We are finally ready to prove Lemma \ref{lem: calculations}
\begin{proof}[Proof of Lemma \ref{lem: calculations}]
Through the proof, we may assume that the conclusion of Lemma \ref{lem: spectral correlations} and Theorem \ref{thm: semi spectral correlations} hold for some fixed $\varepsilon>0$ and $\ell\geq 6$. Moreover, we will use the notation introduced in the proof of Lemma \ref{lem: moments of cos}. By the definition of $s_n$ and Lemma \ref{lem: moments of cos} part $(1)$,	we have 
	\begin{align}
\mathcal{S}^{-1}	\int_{B(z,s)\cap [0,1]^2} s_n(x)dx= \frac{1}{N}\sum_{\xi\backslash\sim}  \mathcal{S}^{-1}\int_{B(z,s)\cap [0,1]^2}  \cos (2\pi \xi_1x_1)+\cos (2\pi \xi_2x_2)dx \nonumber \\
 -\frac{1}{N}\sum_{\xi\backslash\sim}  \mathcal{S}^{-1}\int_{B(z,s)\cap [0,1]^2} \cos (2\pi \xi_1x_1)\cos (2\pi \xi_2x_2 ) dx \nonumber \\
\ll  \frac{1}{N}\sum_{\xi\backslash\sim}\frac{1}{|\xi_1 s|} \ll n^{-\varepsilon/2},
	\end{align}
this proves $(1)$. 

Invoking Lemma \ref{lem: moments of cos} parts $(2)$ and $(3)$, we see that 
	\begin{align}
		&\mathcal{S}^{-1}\int_{B(z,s)\cap [0,1]^2} s_n^2(x)=  \frac{5}{N} + O(n^{-\epsilon/2}) \label{10} \\
		& \mathcal{S}^{-1}\int_{B(z,s)\cap [0,1]^2} b_{11}(x)dx \ll Nn^{1-\epsilon/2} \label{8.1} \\
			& \mathcal{S}^{-1}\int_{B(z,s)\cap [0,1]^2} b_{22}(x)dx \ll N n^{1-\epsilon/2},\label{8.2}
	\end{align}
this proves $(3)$. 

We now begin the proof of $(2)$, first we	observe that
	\begin{align}
		\sum_{\xi\backslash\sim}\xi_i \xi_j = \frac{n}{2}\frac{N}{4}\delta_{ij}. \label{8.3}
	\end{align}
Moreover,	separating diagonal terms from the off-diagonal ones, using $2\sin(a)^2=1-\cos(2a)$, $8\sin(a)^4= 3- 3\cos(2a)$  and Lemma \ref{lem: moments of cos} part $(1)$, we obtain 
	\begin{align}
		&\mathcal{S}^{-1}\int_{B(z,s)\cap [0,1]^2}d_1^2(x)dx \nonumber \\	
		=&\mathcal{S}^{-1}\int_{B(z,s)\cap [0,1]^2}	\sum_{\xi,\eta} \xi_1\eta_1 \sin(2\pi \xi_1 x_1) \sin(2\pi \eta_1 x_1) \sin(\pi \xi_2 x_2)^2\sin(\pi \eta_2 x_2)^2 dx \nonumber \\
		=&   \frac{3}{16}\sum_{\xi\backslash\sim} \xi_1^2 + O \left(	\sum_{\xi\neq\eta} \xi_1\eta_1 s\cdot \mathcal{S}^{-1}\left|\int_{a_1s} ^{b_1s}\sin(2\pi \xi_1 x_1) \sin(2\pi \eta_1 x_1)dx_1\right| + \max_{t=1,2} \sum_{\xi\backslash\sim} \frac{n}{s|\xi_t|}\right). \label{4.1.1}
	\end{align}
Thus, using $2\sin(a)\sin(b)= \cos(a-b)+ \cos(a+b)$, Lemma \ref{lem: moments of cos} part $(1)$  and Theorem \ref{thm: semi spectral correlations}, we obtain 
\begin{align}
		\mathcal{S}^{-1}\int_{B(z,s)\cap [0,1]^2}d_1^2(x)dx= 	\frac{3nN}{2^7} + O\left(N\sum_{\xi\neq\eta} \frac{n}{s|\xi_1-\eta_1|} + Nn^{1-\varepsilon/2}\right)= \frac{3nN}{2^7} + O(N n^{1-\epsilon/2}). \label{9}
	\end{align}
	Similarly we get  
	\begin{align}
		\mathcal{S}^{-1}\int_{B(z,s)\cap [0,1]^2}d_2^2(x)dx= \frac{3nN}{2^7} + O(N n^{1-\epsilon/2}).  \nonumber
	\end{align}
	Observe that similar computations to \eqref{9} give $\mathcal{S}^{-1}\int_{B(z,s)\cap [0,1]^2}d_i^4(x)\ll n^2 N^2$ for $i=1,2$, therefore using the Cauchy-Schwartz inequality and \eqref{10}, we get  
	\begin{align}
		\mathcal{S}^{-1}\int_{B(z,s)\cap [0,1]^2}d_1^2(x)s_n(x)dx= 	\mathcal{S}^{-1}\int_{B(z,s)\cap [0,1]^2}d_2^2(x)s_n(x)dx  \ll n \label{13}
	\end{align}
	Using \eqref{8.2} and \eqref{8.3}, the bound \eqref{13} and the expansion \footnote{Since $|\Gamma_n(x)|=O(1)$, using Lemma \ref{lem: singular set} together with Theorem \ref{thm: semi spectral correlations} as in the proof of Theorem \ref{thm:1}, we may assume that $x\in B(z,s)\cap [0,1]^2 \backslash \mathcal{Q}_{\text{sing}}$.  } $\Var f(x)^{-1}= 1+ O(s_n(x))$, we get 
	\begin{align}
		\mathcal{S}^{-1}\int_{B(z,s)\cap [0,1]^2} \Tr (\Gamma_n(x))dx &= - \frac{2^7\mathcal{S}^{-1}}{n N^2}\int_{B(z,s)\cap [0,1]^2}\frac{1}{\Var f(x)}[ d_1^2(x)+ d_2^2(x)]dx +O\left( n^{-\epsilon}\right) \nonumber \\
		&= - \frac{2^7\mathcal{S}^{-1}}{n N^2}\int_{B(z,s)\cap [0,1]^2}[ d_1^2(x)+ d_2^2(x)] [1+O(s_n(x))]dx + O\left( N^{-2}\right) \nonumber \\
		&= -\frac{6}{N} + O\left(\frac{1}{N^2}\right),
	\end{align}
this concludes the proof of $(2)$. 

We are now going to prove $(4)$. First, we observe that 
	\begin{align}
	&\mathcal{S}^{-1}\int_{B(z,s)\cap [0,1]^2} s_n(x)b_{11}(x) dx=\mathcal{S}^{-1}\int_{B(z,s)\cap [0,1]^2} s_n(x)b_{11}(x) dx\nonumber \\
	&= \frac{4}{N}\sum_{\xi,\eta}\mathcal{S}^{-1}\int_{B(z,s)\cap [0,1]^2} \xi_1^2 (\cos (2\pi \xi_1x_1)-\cos (2\pi \xi_2x_2)- \cos (2\pi \xi_1x_1)\cos (2\pi \xi_2x_2) ) \cdot \nonumber \\
	&\cdot (\cos (2\pi \eta_1x_1)+\cos (2\pi \eta_2x_2)- \cos (2\pi \eta_1x_1)\cos (2\pi \eta_2x_2) )dx
	\end{align}
Using Lemma \ref{lem: moments of cos} parts $(1)$, $(2)$, $(3)$ and \eqref{8.3}, we have
	\begin{align}
			\mathcal{S}^{-1}\int_{B(z,s)\cap [0,1]^2} s_n(x)b_{11}(x) dx=  \frac{4}{N}\sum_{\xi\backslash\sim} \frac{\xi_1^2}{4}+ O\left(n^{1-\epsilon}\right) = \frac{n}{8} +O\left(n^{1-\epsilon/2}\right) \label{8.5} 
		\end{align}
		Therefore, since
		\begin{align}
		\mathcal{S}^{-1}\int_{B(z,s)\cap [0,1]^2} s_n(x)\cdot \Tr \Gamma_n(x)dx&= \frac{8}{n N}\mathcal{S}^{-1}\int_{B(z,s)\cap [0,1]^2}s_n(x)[b_{11}(x)+ b_{22}(x)]dx - \nonumber \\
		& \frac{2^7}{n N^2}\mathcal{S}^{-1}\int_{B(z,s)\cap [0,1]^2}\frac{s_n}{\Var(f(x))}[d_1^2(x)+d_2^2(x)] dx ,
		\end{align}
		the bound \eqref{13} together with the asymptotic relation \eqref{8.5} give
		\begin{align}
		\mathcal{S}^{-1}\int_{B(z,s)\cap [0,1]^2} s_n(x)\cdot \Tr \Gamma_n(x)dx = \frac{2}{N} + O\left( N^{-2}\right). \nonumber
		\end{align}
	This concludes the proof of $(4)$. 
	
To prove $(5)$, upon recalling that $\Var (f(x))^{-1/2}= 1+O(s_n(x))$, we observe that 
		\begin{align}
	[\Tr (\Gamma_n(x))]^2&= \frac{8^2}{n^2 N^2}[b_{11}^2(x) +b_{22}^2(x)+ 2 b_{11}(x)b_{22}(x)] \nonumber \\
	&+ \frac{128^2}{n^2N^4}[d_1^4(x)+ d_2^4(x) +2d_1^2(x)d_2^2(x)][1+O\left(s_n(x)\right)] \nonumber \\
	&- 2\frac{8}{n N}\frac{128}{n N^2}[b_{11}(x) +b_{22}(x)][d_1^2(x)+ d_2^2(x)][1+ O\left(s_n(x)\right)]. \label{4.2.0}
\end{align}	
		Using Lemma \ref{lem: moments of cos} parts $(2)$ and $(3)$ and Theorem \ref{thm: semi spectral correlations}, we have 
		\begin{align}
			&	\mathcal{S}^{-1}\int_{B(z,s)\cap [0,1]^2}b_{11}^2(x)dx= 	\mathcal{S}^{-1}\int_{B(z,s)\cap [0,1]^2} b_{22}^2(x) dx= \frac{5}{4}\sum_{\xi\backslash\sim} \xi_1^4 + O\left(n^{1-\epsilon/2} \right) \nonumber \\
				&\mathcal{S}^{-1}\int_{B(z,s)\cap [0,1]^2}b_{22}(x)b_{11}(x)dx=-\frac{3}{4}\sum_{\xi\backslash\sim} \xi_1^2\xi_2^2  + O\left(n^{1-\epsilon/2} \right) \label{4.1.2}
	\end{align}
Moreover, we observe that
	\begin{align}
			\mathcal{S}^{-1}\int_{B(z,s)\cap [0,1]^2} d_1^4(x)dx= \sum_{\xi^1,\xi^2,\xi^3,\xi^4} \mathcal{S}^{-1}\int_{B(z,s)\cap [0,1]^2}\prod_{i=1}^4 \xi_1^i\sin (2\pi \xi_1^i x_1)\sin (2\pi \xi_2^i x_2)^2. \label{4.1.3}
	\end{align}
Thus,	separating the terms with $\xi^1=...=\xi^4$ and the terms with $\xi^1=\xi^2$ and $\xi^3=\xi^4$ from the rest, arguing as in \eqref{4.1.1} and bearing in mind that $|\xi_t|\leq n^{1/2}$, we obtain 
	\begin{align}
			\mathcal{S}^{-1}\int_{B(z,s)\cap [0,1]^2} d_1^4(x)dx &= \frac{105}{1024}\sum_{\xi\backslash\sim} \xi_1^4 + \frac{9}{256}\sum_{\xi\neq\eta}\xi_1^2 \eta_1^2 + O\left( \sum_{\substack{\xi^1,...,\xi^4\\ \sum_i \xi_1^i\neq 0}} \frac{n^2}{s|\sum_i \xi_1^i|} \right)  \nonumber \\
			&= \frac{105}{1024}\sum_{\xi\backslash\sim} \xi_1^4 + \frac{9}{256}\sum_{\xi\neq\eta}\xi_1^2 \eta_1^2 + O\left( n^{2-\epsilon/2}\right), \nonumber
	\end{align}
where, in the last line, we have used Theorem \ref{thm: semi spectral correlations}.	Similar computations give 
	\begin{align}
		&	\mathcal{S}^{-1}\int_{B(z,s)\cap [0,1]^2} d_2^4(x)dx=\frac{105}{1024}\sum_{\xi\backslash\sim} \xi_1^4 + \frac{9}{256}\sum_{\xi\neq\eta}\xi_1^2 \eta_1^2 + O\left( n^{2-\epsilon/2}\right) \nonumber \\
		& \mathcal{S}^{-1}\int_{B(z,s)\cap [0,1]^2} d_1^2(x)d_2^2(x)dx =\frac{25}{1024}\sum_{\xi\backslash\sim} \xi_1^2\xi_2^2 + \frac{9}{256}\sum_{\xi\neq\eta}\xi_2^2 \eta_1^2 + O\left( n^{2-\epsilon/2}\right) \nonumber \\
		&	\mathcal{S}^{-1}\int_{B(z,s)\cap [0,1]^2} b_{ii}(x)d_j^2(x)dx \ll n^2 N^2 \hspace{5mm}i=1,2 \label{4.1.4}
		\end{align}
	Finally, bearing in mind that $n= \xi_1^2+\xi_2^2$ and $\sum_{\xi\backslash\sim}\xi_1^2= nN/8$, we have  
$$\sum_{\xi\backslash\sim}\xi_1^2\xi_2^2= \frac{n^2 N}{8} - \sum_{\xi\backslash\sim}\xi_1^{4}$$
and, for $i,j=1,2$,
$$\sum_{\xi \neq \eta}\xi_i^2\eta_j^2 = \sum_{\xi }\xi_i^2\sum_{\eta}\left(\eta_j^2- \xi_i^2\right)=  \frac{n^2 N}{8} - \sum_{\xi\backslash\sim}\xi_1^{4}.$$
		Thus, $(5)$ follows inserting \eqref{4.1.2}, \eqref{4.1.3} and \eqref{4.1.4} into \eqref{4.2.0}. 

To see $(6)$ we observe that, for symmetric matrix $A=a_{ij}$, $i,j=1,2$,  $\Tr(A^2)= a_11^2+ 2a_{12}^2+ a_22^2$, thus 
\begin{align}
	\Tr(\Gamma_n^2(x))= \left(\frac{8}{nN}b_{11}(x)- \frac{128}{nN^2\Var(f(x)) d_1^2(x)}\right)^2 \nonumber \\
	+ 2\left(\frac{8}{nN}b_{12}(x)- \frac{128}{nN^2\Var(f(x)) d_1(x)d_2(x)}\right)^2 + \left(\frac{8}{nN}b_{22}(x)- \frac{128}{nN^2\Var(f(x)) d_2^2(x)}\right)^2. \label{5.0}
\end{align}		
Thanks to Lemma \ref{lem: moments of cos} parts $(2)$ and $(3)$, we have 
\begin{align}
\mathcal{S}^{-1}\int_{B(z,s)\cap [0,1]^2} b_{12}^2(x)dx = \frac{1}{4}\sum_{\xi\backslash\sim} \xi_1^2\xi_2^2 + O\left( \sum_{\xi \neq \eta} \frac{n^2}{s|\xi_t-\eta_t|}\right) \label{5.1}
\end{align}
and 
$$ \mathcal{S}^{-1}\int_{B(z,s)\cap [0,1]^2} b_{12}(x)d_1(x)d_2(x)dx \ll n^2 N.$$
Therefore, part $(6)$ follows from inserting \eqref{4.1.2}, \eqref{4.1.4} and \eqref{5.1} into \eqref{5.0}. 
		Finally, separating diagonal terms from the off-diagonal ones, we observe that 
		\begin{align} 
		& \mathcal{S}^{-1}\int_{B(z,s)\cap [0,1]^2}s_n(x)^3dx= \frac{4^3}{N^3}\sum_{\xi\backslash\sim}\left(-\frac{3}{4}\right) + O\left(\frac{1}{N^3} \sum_{\xi^1,\xi^2,\xi^3}\frac{1}{s|\xi_i^1+\xi_i^2+\xi_i^3|}\right) \ll N^{-2}, \nonumber
	\end{align}
where in the last line we have used Theorem \ref{thm: semi spectral correlations}. Similarly, we have 
$$\mathcal{S}^{-1}\int_{B(z,s)\cap [0,1]^2}\Tr (\Gamma_n(x))^3dx\ll N^{-2}.$$
Thus, we have proved parts $(7)$ and $(8) $, and hence Lemma \ref{lem: calculations}. 
\end{proof}

\bibliographystyle{siam}
\bibliography{ba}

\begin{thebibliography}{10}

\bibitem{AW09}
{\sc J.~Azais and M.~Wschebor}, {\em Level Sets and Extrema of Random Processes
  and Fields}, Wiley, New York, 2009.

\bibitem{BMW}
{\sc J.~Benatar, D.~Marinucci, and I.~Wigman}, {\em Planck-scale distribution
  of nodal length of arithmetic random waves}, J. Anal. Math., 141 (2020),
  pp.~707--749.

\bibitem{B1}
{\sc M.~V. Berry}, {\em Regular and irregular semiclassical wavefunctions},
  Journal of Physics A: Mathematical and General, 10 (1977), p.~2083.

\bibitem{B2}
\leavevmode\vrule height 2pt depth -1.6pt width 23pt, {\em Semiclassical
  mechanics of regular and irregular motion}, Les Houches lecture series, 36
  (1983), pp.~171--271.

\bibitem{BB}
{\sc E.~Bombieri and J.~Bourgain}, {\em A problem on sums of two squares}, Int.
  Math. Res. Not. IMRN,  (2015), pp.~3343--3407.

\bibitem{B78}
{\sc J.~Br\"{u}ning}, {\em \"{U}ber {K}noten von {E}igenfunktionen des
  {L}aplace-{B}eltrami-{O}perators}, Math. Z., 158 (1978), pp.~15--21.

\bibitem{BG72}
{\sc J.~Br\"{u}ning and D.~Gromes}, {\em \"{U}ber die {L}\"{a}nge der
  {K}notenlinien schwingender {M}embranen}, Math. Z., 124 (1972), pp.~79--82.

\bibitem{CKW20}
{\sc V.~Cammarota, O.~Klurman, and I.~Wigman}, {\em Boundary effect on the
  nodal length for arithmetic random waves, and spectral semi-correlations},
  Comm. Math. Phys., 376 (2020), pp.~1261--1310.

\bibitem{CMW20}
{\sc V.~Cammarota, D.~Marinucci, and I.~Wigman}, {\em Nodal deficiency of
  random spherical harmonics in presence of boundary}, J. Math. Phys., 62
  (2021), pp.~022701, 20.

\bibitem{C76}
{\sc S.~Y. Cheng}, {\em Eigenfunctions and nodal sets}, Comment. Math. Helv.,
  51 (1976), pp.~43--55.

\bibitem{DF}
{\sc H.~Donnelly and C.~Fefferman}, {\em Nodal sets of eigenfunctions on
  reimannian manifolds}, Inventiones mathematicae, 93 (1988), pp.~161--183.

\bibitem{DF90}
\leavevmode\vrule height 2pt depth -1.6pt width 23pt, {\em Nodal sets of
  eigenfunctions: {R}iemannian manifolds with boundary}, in Analysis, et
  cetera, Academic Press, Boston, MA, 1990, pp.~251--262.

\bibitem{FI}
{\sc J.~Friedlander and H.~Iwaniec}, {\em Opera de cribro}, vol.~57 of American
  Mathematical Society Colloquium Publications, American Mathematical Society,
  Providence, RI, 2010.

\bibitem{K}
{\sc I.~Kubilyus}, {\em The distribution of {G}aussian primes in sectors and
  contours}, Leningrad. Gos. Univ. U\v{c}. Zap. Ser. Mat. Nauk, 137(19) (1950),
  pp.~40--52.

\bibitem{L2}
{\sc A.~Logunov}, {\em Nodal sets of {L}aplace eigenfunctions: polynomial upper
  estimates of the {H}ausdorff measure}, Ann. of Math. (2), 187 (2018),
  pp.~221--239.

\bibitem{L1}
\leavevmode\vrule height 2pt depth -1.6pt width 23pt, {\em Nodal sets of
  {L}aplace eigenfunctions: proof of {N}adirashvili's conjecture and of the
  lower bound in {Y}au's conjecture}, Ann. of Math. (2), 187 (2018),
  pp.~241--262.

\bibitem{LM}
{\sc A.~Logunov and E.~Malinnikova}, {\em Nodal sets of {L}aplace
  eigenfunctions: estimates of the {H}ausdorff measure in dimensions two and
  three}, in 50 years with {H}ardy spaces, vol.~261 of Oper. Theory Adv. Appl.,
  Birkh\"{a}user/Springer, Cham, 2018, pp.~333--344.

\bibitem{S20}
{\sc A.~Sartori}, {\em Planck-scale number of nodal domains for toral
  eigenfunctions}, J. Funct. Anal., 279 (2020), pp.~108663, 22.

\bibitem{Tbook}
{\sc G.~Tenenbaum}, {\em Introduction to analytic and probabilistic number
  theory}, vol.~163 of Graduate Studies in Mathematics, American Mathematical
  Society, Providence, RI, third~ed., 2015.
\newblock Translated from the 2008 French edition by Patrick D. F. Ion.

\end{thebibliography}

\end{document}